\documentclass[12pt]{amsart}
\usepackage{amsfonts, amssymb,enumerate}

\usepackage[margin=0.8in]{geometry}
\usepackage[all,cmtip]{xy}
\usepackage{tikz}
\usepackage{amsmath}
\usepackage{amsfonts}
\usepackage{amsthm}
\usepackage{amssymb}
\usepackage{times} 
\usepackage{graphicx}
\usepackage{fixltx2e}
\usepackage{url}
\usepackage{MnSymbol}
\usetikzlibrary{arrows}

\newtheorem{thm}{Theorem}[section]
\newtheorem*{thm*}{Theorem}
\newtheorem{fact}[thm]{Fact}
\newtheorem{cor}[thm]{Corollary}
\newtheorem*{cor*}{Corollary}
\newtheorem{lem}[thm]{Lemma}
\newtheorem{prop}[thm]{Proposition}
\theoremstyle{remark}
\newtheorem{remark}[thm]{Remark}

\newtheorem{defn}{Definition}
\newtheorem*{example}{Example}
\theoremstyle{definition}

\tikzset{node distance=4cm, auto}

\begin{document}
\title{Motivic Decomposition of Projective Pseudo-homogeneous Varieties}
\author{Srimathy Srinivasan}
\address{Department of  Mathematics, University of Maryland, College Park, MD 20705, USA}
\email{srimathy@math.umd.edu}
\date{}

\begin{abstract}
Let $G$ be a semi-simple algebraic group over a perfect field  $k$. A lot of progress has been made recently in computing the Chow motives of projective $G$-homogeneous varieties. When $k$ has positive characteristic, a broader class of $G$-homogeneous varieties appear.  These  are varieties over which $G$ acts transitively with possibly non-reduced isotropy subgroup. In this paper we study these varieties which we call {\it projective pseudo-homogeneous varieties}  for $G$ inner type over $k$ and prove that their motives satisfy  Rost nilpotence.  We also find  their motivic decompositions  and relate them to the motives of  corresponding homogeneous varieties.  
\end{abstract}
\maketitle
\section{Introduction}
Let $G$ be a semi-simple  algebraic group of inner type over a perfect field $k$ of characteristic $p >3$ (See Remark \ref{rmk:char} for why the assumption $p>3$ is necessary).  We follow the terminology of SGA3. So by definition $G$ is smooth and connected with trivial radical.   Note that in SGA3, parabolic subgroups are reduced as schemes.  Therefore we use the term \textit{parabolic subgroup schemes} to  include possibly non-reduced subgroup schemes containing a Borel. Let $K$ denote the algebraic closure of $k$. For a variety $Y$ over $k$ and an extension $k'\supseteq k$, we write $Y_{k'}$ for $Y \times_{Spec~k} Spec~k'$. 
\begin{defn}  A $G$-variety $\widetilde{X}$  over $k$ is called a \textit{projective pseudo-homogeneous variety} \footnote{The term projective pseudo-homogeneous varieties  is coined in this paper to point towards a natural generalization of projective homogeneous varieties.  It is not to be confused with the definition used in \cite{karpenko_pseudo}} if $\widetilde{X}_K  \simeq G_K/\widetilde{P}$ for some parabolic subgroup scheme $\widetilde{P}$ in $G_K$ that is not necessarily reduced.
 \end{defn}

Such a variety is always smooth since $G$ is smooth (See SGA3, exp VI\textsubscript{A}, Theorem 3.2). For detailed construction of the quotient of  an algebraic group by a subgroup see Chapter III, \S3 of \cite{DG}. Note that  by Proposition 2.1, \S3, Chapter III of \cite{DG}, the condition $\widetilde{X}_K  \simeq G_K/\widetilde{P}$ is equivalent to saying that the action map $G(\Omega) \times \widetilde{X}(\Omega) \rightarrow \widetilde{X}(\Omega) \times \widetilde{X}(\Omega)$ is surjective for every algebraically closed field $\Omega$ over $K$. If $\widetilde{P}$ is a parabolic subgroup scheme over $K$, we will make slight abuse of notation and write $G/\widetilde{P}$ for $G_K/\widetilde{P}$.  Let $P$ denote the underlying reduced scheme of $\widetilde{P}$. Note that since $k$ is perfect, $P$ is  a group scheme (See \S6 in Chapter VI of \cite{AGSMilne}).
 
 \begin{defn} Given $\widetilde{X}$,  a projective pseudo-homogeneous variety for $G$ such that $\widetilde{X}_K  \simeq G/\widetilde{P}$,  let  $X$ denote the unique (see Proposition 1.3 in \cite{index_red})  projective homogeneous variety for $G$, such that $X_K \simeq G/P$ where $P$ is the underlying reduced subscheme of $\widetilde{P}$. We call $X$  \textit{the projective homogeneous variety corresponding to} $\widetilde{X}$.
 \end{defn}
 
 By universal property of quotients, there is a canonical $G$-equivariant finite morphism $\theta: X \rightarrow \widetilde{X}$. 
 
\begin{example}  Suppose $G=SL_{3,k}$. Let $G/\widetilde{P} \subseteq \mathbb{P}^2 \times \mathbb{P}^2$ be given by the equation $\sum_{i=0}^2 x_i^{p}y_i=0$  where the $G$ action is   $g.\overrightarrow{x} = g^{p^3}\overrightarrow{x}$ and $g.\overrightarrow{y} = (g^{-t })^{p^4}\overrightarrow{y}$  (Here $g^{-t} = (g^{-1})^t$ is the transpose of the inverse of $g$.  Also by abuse of notation $g^{p^n}$ means taking $p^n$th  power of entries of the matrix $g$). Then $\widetilde{P} = Stab([1:0:0] \times [0:0:1]) = \{\left( \begin{smallmatrix} *&*&*\\ x&*&* \\ y&z&*\end{smallmatrix} \right) | x^{p^3} = 0, y^{p^3} = 0, z^{p^4} = 0\}$. The underlying reduced scheme is the standard Borel $ P = \left( \begin{smallmatrix} *&*&*\\ 0&*&* \\ 0&0&*\end{smallmatrix} \right)$ and the corresponding homogeneous variety $G/P \subseteq \mathbb{P}^2 \times \mathbb{P}^2$ is given by $ \sum_{i=0}^2 x_i y_i =0$.  This comes with the standard $G$-action $g.\overrightarrow{x} = g\overrightarrow{x}$ and $g.\overrightarrow{y} = (g^{-t })\overrightarrow{y}$.   We have the canonical $G$-equivariant map
\begin{eqnarray*}
G/{P} \rightarrow G/\widetilde{P} \\
\overrightarrow{x} \rightarrow \overrightarrow{x}^{p^3}\\
\overrightarrow{y} \rightarrow \overrightarrow{y}^{p^4}
\end{eqnarray*}
\end{example}

 We want to emphasize that by Theorem 5.2 in \cite{splitting_lauritzen}, the $K$-varieties $G/\widetilde{P}$ and $G/P$ are not in general isomorphic.  Therefore, $X$ and $\widetilde{X}$ need not be twisted forms of each other.\\
\indent In this paper we prove that Rost nilpotence theorem holds for projective pseudo-homogeneous varieties. We also compute the Chow motives of these varieties and show that their motives are isomorphic to motives of the corresponding projective homogeneous varieties.  A crucial ingredient of the proof is Theorem \ref{thm:cool} which gives a characterization of when the motive of a variety is isomorphic to the motive of a projective homogeneous variety.  The proof of this theorem is independent of the characteristic of the base field and might be useful for other applications.

\subsection{Notations}
Throughout this paper $k$ is a perfect field of characteristic $p>3$ and  $K$ denotes the algebraic closure of $k$. $\mathbb{G}_m$ denotes the usual multiplicative group.  $G$ denotes a semi-simple algebraic group of inner type over $k$.   The set of vertices  of the Dynkin diagram of $G$ (or equivalently the set of conjugacy classes of maximal parabolics in $G_K$) is denoted by $\Delta_G$. For a field extension $E$ of $k$,  $\tau_{E} \subseteq \Delta_G$ denotes the subset  that contain  the classes of those maximal parabolics in $G_K$  defined over $E$.  Given a parabolic subgroup scheme $\widetilde{P}$,  $P$ denotes the underlying reduced subscheme.   If $\widetilde{X}$ is a projective  pseudo-homogeneous variety then $X$ denotes the  corresponding projective homogeneous variety.\\
\indent $\Lambda$ denotes a  connected, finite, associative unital  commutative ring. An example to keep in mind is a finite field of some prime characteristic.   Let $Chow(k, \Lambda)$ denote the category of Chow motives over $k$  with coefficients in $\Lambda$. Detailed exposition of $Chow(k, \Lambda)$ can be found in \cite{elman_quadratic}. For a variety $X$, $\mathcal{M}(X)$ denotes the Chow motive of $X$. By $Ch_i(X)$ and $Ch^i(X)$ we mean the $i^{th}$ Chow group of $X$ graded by dimension and codimension respectively.   The Tate motive $\mathcal{M}(Spec~k)\{i\}$ is denoted by $\Lambda\{i\}$ (The notation  $\Lambda\{i\}$ is equal to $\Lambda(i)[2i]$  in Voevodsky's category of motives). For a motive $M$, $M\{i\}:=M \otimes \Lambda\{i\}$. \\

 \subsection{Statements of Main Results}
We say that \emph{Krull-Schmidt principle} holds for an object in an additive category if it  is isomorphic uniquely to direct sum of indecomposable summands (up to permutation). Let $X$ be a $k$-variety. Recall from Karpenko's paper \cite{upperkarpenko} that a summand $M$ of  $\mathcal{M}(X)$ is called \emph{upper} if $Ch^0(M) \neq 0$. See Lemma 2.8 in \cite{upperkarpenko} for more details. If the motive of $X$  satisfies Krull-Schmidt principle, let $U_X$ denote the unique upper indecomposable summand of $\mathcal{M}(X)$. It is well known that the motives of projective homogeneous varieties satisfy Krull-Schmidt principle (see Corollary \ref{cor:krull_homo}) in $Chow(k, \Lambda)$.  If $X_{\tau}$ is projective homogeneous corresponding to the subset $\tau \subseteq \Delta_G$ (see \S\ref{sec:homog}), we write $U_{\tau}$ for the upper indecomposable summand of  $\mathcal{M}(X_{\tau})$.

\begin{thm}\label{thm:nilpotence}(Rost Nilpotence for Projective Pseudo-Homogeneous Varieties)
Let $\widetilde{X}$ be a projective pseudo-homogeneous variety for a semi-simple group $G$ of inner type over $k$.  Then the kernel of the base change map
\begin{equation*}
\begin{split}
End(\mathcal{M}(\widetilde{X}))  &\rightarrow   End(\mathcal{M}(\widetilde{X}_K))\\
f   &\mapsto  f \otimes K
\end{split}
\end{equation*}
consists of nilpotents. 
\end{thm}
\begin{proof} See \S \ref{sec:rost}.
\end{proof}

\begin{thm}\label{thm:krull_pseudo}
The Krull-Schmidt principle holds for any shift of any summand of the motive of a projective pseudo-homogeneous  variety for $G$.
\end{thm}
\begin{proof}
This follows from Theorem  \ref{thm:nilpotence},  Theorem \ref{thm:krull} and Theorem  \ref{thm:rank}.
\end{proof}

The following theorem gives a characterization of when the motive of a variety is isomorphic to the motive a projective homogeneous variety and is independent of the characteristic of the base field $k$. In particular, it holds for characteristic zero as  well. Recall  that a $k$-variety $Z$  is geometrically split if $\mathcal{M}(Z_K)$ is isomorphic to a direct sum of Tate motives.
\begin{thm}
\label{thm:cool}
 Let $X$ be projective $G$-homogeneous variety  over $k$. Let $Z$ be any  geometrically split projective $k$-variety whose motive satisfies the Rost nilpotence principle such that the following holds in $Chow(k, \Lambda)$:
\begin{enumerate}
\item $U_X \simeq U_Z$
\item $\mathcal{M}(X_L) \simeq \mathcal{M}(Z_L)$ where $L = k(X)$
\end{enumerate}
Then $\mathcal{M}(X) \simeq \mathcal{M}(Z)$.
\end{thm}
\begin{proof} See \S \ref{sec:final1}.
\end{proof}

\begin{remark}
In the above theorem,  $\mathcal{M}(Z)$ satisfies Krull-Schmidt by Theorem \ref{thm:krull} and hence the upper motive $U_Z$ of $Z$ is well-defined.
\end{remark}

As an application of the above theorem we derive the following main result.
\begin{thm}
\label{thm_motives}
Let  $\widetilde{X}$ be a projective pseudo-homogeneous variety for $G$ and let $X$ be the corresponding projective homogeneous variety. Then in the category of motives $Chow(k,\Lambda)$
\begin{eqnarray*}
\mathcal{M}(X)\simeq \mathcal{M}(\widetilde{X})
\end{eqnarray*}
In particular, by Theorem \ref{thm:karpenko}  every indecomposable summand in $\mathcal{M}(\widetilde{X})$ is a shift of some upper motive $U_\tau$  satisfying $\tau_{k(X)} \subseteq \tau$.
\end{thm}
\begin{proof} See \S \ref{sec:final2}.
\end{proof}

Let $A$ be a central simple algebra of degree $n$ over $k$.   Let $X=X(d_1, d_2, \cdots, d_m , A)$ be the variety of right ideals of reduced dimensions $1 \leq d_1 <d_2 < \cdots< d_m \leq n$. Note that $X$ is projective homogeneous for $G= PGL(A)$. Write $X_K \simeq G/P$ for some parabolic subgroup $P$.   Let $A^{(p)} = A \otimes_{Fr} k$  and $X^{(p)} = X \times _{Fr} Spec~k$ where $Fr: k \rightarrow k$ is the Frobenius morphism. Then it is easy to see that ${X^{(p)}}_K \simeq G/\widetilde{P}$ where $\widetilde{P} = G_{p} P$ and $G_{p}$ is the kernel of the Frobenius morphism $Fr: G \rightarrow G^{(p)}$. Moreover, $X$ is  the projective homogeneous variety corresponding to $X^{(p)}$.\\
\indent An easy consequence of Theorem \ref{thm_motives} is the following.
\begin{cor}
\label{cor:csa}
 For a central simple algebra $A$ over $k$ of degree $n$, let $B$ denote the central simple algebra of degree $n$ that is  Brauer equivalent to $A^{\otimes p}$. Then  in the category $Chow(k, \Lambda)$,  the motives of twisted flag varieties $X(d_1, d_2, \cdots, d_m , A)$ and $X(d_1, d_2, \cdots, d_m , B)$ are isomorphic. That is, 
\begin{align*}
\mathcal{M}(X(d_1, d_2, \cdots, d_m , A)) \simeq \mathcal{M}(X(d_1, d_2, \cdots, d_m , B))
\end{align*}
Taking $m=1$,  we get $ \mathcal{M}(SB_d(A)) \simeq \mathcal{M}(SB_d(B))$ for twisted Grassmannians. In particular, for the case of  Severi-Brauer varieties we have $ \mathcal{M}(SB(A)) \simeq \mathcal{M}(SB(B))$.
\end{cor}
\begin{proof}
Note that  $B = A^{(p)}$  by  Theorem 3.9 in \cite{brauer_p} (see also Proposition 3.2 in  \cite{florence}).  Therefore, 
\begin{align*}
\mathcal{M}(X(d_1, d_2, \cdots, d_m , B)) &\simeq \mathcal{M}(X(d_1, d_2, \cdots, d_m , A^{(p)}))\\
&\simeq \mathcal{M}(X(d_1, d_2, \cdots, d_m ,A)^{(p)}) ~~~~~~~~(\text{by functoriality of the Frobenius})\\
&\simeq \mathcal{M}(X(d_1, d_2, \cdots, d_m ,A)) ~~~~~~~~~~~~~(\text{by Theorem \ref{thm_motives}})
\end{align*}

The rest follows easily.
\end{proof}

\begin{remark} Let $A$ be a central simple algebra over $k$   with exponent (i.e., the order of its Brauer class as an element in the  Brauer group) not dividing $p^2 -1$.  Let $X = SB(A)$ be the Severi-Brauer variety associated with $A$ and let $X^{(p)} = SB(A)^{(p)} \simeq SB(A^{(p)})$.  Then by Corollary \ref{cor:csa},  $\mathcal{M}(X)$ and $\mathcal{M}(X^{(p)})$ are isomorphic in $Chow(k, \Lambda)$ for all coefficient  rings $\Lambda$ that are finite fields (of any characteristic). But they are not isomorphic in the integral Chow motive category  $Chow(k, \mathbb{Z})$. Indeed, if they were isomorphic in $Chow(k, \mathbb{Z})$,  Criterion 7.1 in \cite{karpenko_mequiv} would imply that  $A^{(p)}$ is isomorphic either  to $A$ or its opposite $A^{op}$ .  Since $A^{(p)}$ is Brauer equivalent to  $A^{\otimes p}$  by Proposition 3.2 in  \cite{florence}, this contradicts our assumption on the exponent of $A$. Therefore we get examples of varieties whose motives are isomorphic over all finite field coefficients but not over integral coefficients. 
\end{remark}

\subsection{Outline}
 In \S\ref{sec:prelim} we briefly recall the  facts known about projective homogeneous varieties and their motives.   In \S\ref{sec:iso} we give  motivic decompositions of projective  pseudo-homogeneous varieties  for isotropic $G$ and relate them to corresponding projective homogeneous varieties. In \S\ref{sec:rost} we prove that Rost nilpotence holds for such varieties. In \S\ref{sec:split}  we study their cellular structure and compute their motives when $G$ is split. Finally, in \S\ref{sec:final} we compute the  motivic decompositions of projective pseudo-homogeneous variety and relate them to the decompositions of corresponding projective homogeneous varieties.

\section{Preliminaries}
\label{sec:prelim}
\indent  Projective pseudo-homogeneous varieties are extensively studied in the literature when  $k = K$ is algebraically closed. We give a brief survey on what is known so far.  In \cite{wenzel}, Wenzel has classified all parabolic subgroup schemes $\widetilde{P}$ and in \cite{rationality} he proved that  the varieties $G/\widetilde{P}$ are rational. Using this classification, de Salas in \cite{desalas}  has classified all $G/\widetilde{P}$. The varieties  of the form $G/\widetilde{P}$ where $\widetilde{P}$ is any parabolic subgroup scheme that may or may not be reduced  are known as  \emph{parabolic varieties}  in \cite{desalas}.     Lauritzen and Haboush answered many interesting questions about the geometry of these varieties including canonical line bundles, vanishing theorems and Frobenius splitting in \cite{lauritzen},  \cite{vufs} and \cite{splitting_lauritzen}.  Lauritzen also gave a geometric construction  of $G/\widetilde{P}$  in \cite{lauritzen_embedding} where he realizes these varieties as  the $G$-orbit of a  Borel stable line in projective space. They  have rich structure and behave quite differently  from the analogous  {\em generalized flag varieties} (or simply \emph{flag varieties})  $G/P$ where $P$ is smooth.  For example, in \cite{splitting_lauritzen}, Lauritzen  has shown that under mild assumptions on $G$, $G/\widetilde{P}$ is isomorphic to a flag variety if and only if $G/\widetilde{P}$ is Frobenius split.  The varieties  of the form $G/\widetilde{P}$  that  it does not admit an isomorphism to a flag variety  are known as \emph{varieties of unseparated flags} or simply \emph{vufs} in \cite{vufs}.   In particular, $G/P$ and $G/\widetilde{P}$ are not isomorphic in general. Moreover, in  \cite{vufs} one can find explicit examples of VUFs  which illustrate  that  unlike  generalized flag varieties, vanishing theorem for ample line bundles and Kodaira's vanishing theorem break down. So over algebraically closed fields,  although these varieties exhibit a lot of strange phenomena, they  are well understood and it is straightforward to compute their Chow motives (see \S\ref{sec:split}).  \\
\indent However, when $k$ is not algebraically closed, nothing much is known about them unlike the analogous  projective homogeneous varieties. Projective homogeneous varieties are quite thoroughly studied in the literature (\cite{artin}, \cite{gille_book},  \cite{elman_quadratic} and \cite{boi}) and so are their  Chow motives  (\cite{bro},  \cite{calmes}, \cite{chernosov_iso}, \cite{outerkarpenko} and  \cite{upperkarpenko}).  Therefore it is natural to study  projective pseudo-homogeneous varieties  and ask if they exhibit any similarity to projective homogeneous varieties.
\subsection{Projective Homogeneous Varieties}
\label{sec:homog}
 In this section we recall some facts known about projective homogeneous varieties. A $G$-variety $X$ is called a \emph{projective homogeneous variety} if $X_K \simeq G/P$ for some parabolic subgroup $P$ (which by definition is smooth). \\
 \indent  The subsets of $\Delta_G$ are in natural one-to-one correspondence with the set of conjugacy classes of  parabolic subgroups in $G_K$  defined as follows: the conjugacy class corresponding $\tau \subseteq \Delta_G$ is the one containing the intersection of all maximal parabolics in $\tau$ that contain a given Borel $B$ in $G_K$. For any subset $\tau \subseteq \Delta_G$, we write $X_{\tau}$ or$X_{\tau, G}$ for the projective homogeneous variety of parabolic subgroups in $G$ of the type $\tau$. For instance, $X_{\Delta_G}$ is the variety of the Borel subgroups. Any projective $G$-homogeneous variety is isomorphic to $X_{\tau}$ for some $\tau$.   Let us recall some of the results known about the motives of projective homogeneous varieties. \\
 \indent In \cite{bro}, Brosnan gave a description about the summands of the motive of projective $G$-homogeneous varieties for isotropic $G$.
\begin{thm}(Corollary 4.1 in \cite{bro})
Let $X$ be a projective $G$-homogeneous variety over $k$. Assume $G$ is isotropic and let $\lambda: \mathbb{G}_m \rightarrow G$ be an embedding of a $k$-split torus.  Then
\begin{align*}
\mathcal{M}(X) =  \coprod \mathcal{M}(Z_i)\{a_i\}
\end{align*}
where $Z_i$ are  connected components of  the fixed point locus $X^{\lambda}$. Moreover,  $Z_i$ are projective homogeneous for the centralizer $H$  of $\lambda$ and  the twists $a_i$ are the dimensions of the positive eigenspace of the action of $\lambda$ on the tangent space of $X$ at an arbitrary point $z \in Z_i$.
\end{thm}

In \cite{bro}, he proved that these varieties also satisfy Rost nilpotence  principle. This is originally due to Chernousov, Gille and Merkurjev (Theorem 8.2 in \cite{chernosov_iso}).
\begin{thm}(Theorem 5.1 in \cite{bro})
\label{thm:rost}
Let $X$ be a projective $G$-homogeneous variety. Then the kernel of the map
\begin{align*}
End(\mathcal{M}(X)) &\rightarrow End(\mathcal{M}(X_K)) \\
f   &\mapsto  f \otimes K
\end{align*}
consists of nilpotent endomorphisms.
\end{thm}

A very useful consequence of Rost nilpotence is the following result which can be found in Karpenko's paper \cite{upperkarpenko}.

\begin{thm}(Corollary 2.6 in \cite{upperkarpenko})
\label{thm:krull}
Assume that the coefficient ring $\Lambda$ is finite. The Krull-Schmidt principle holds for any shift of any summand of the motive of any geometrically split variety in $Chow(k, \Lambda)$  that satisfies  Rost nilpotence principle. 
\end{thm}

A very useful technique to decompose a  motive is due to Rost (\cite{rost_pfister}) and  Karpenko (\cite{cellular_karpenko}). We state  this below for convenience of the reader.

\begin{thm}(\cite{chernosov_iso}, \cite{delBano}, \cite{cellular_karpenko})
\label{thm:cellular}
Let X be a smooth, projective variety over a field $k$ with a filtration
\begin{align*}
X = X_n \supseteq  X_{n-1} \supseteq  \cdots \supseteq X_0\supseteq  X_{-1} =\emptyset
\end{align*}
where the $X_i$ are closed subvarieties. Assume that, for each integer $i \in [0,n]$, there is a smooth projective variety $Z_i$ and an affine fibration $\phi_i: X_i - X_{i-1} \rightarrow Z_i$ of relative dimension $a_i$.  Then, in the category of correspondences,
\begin{align*}
\mathcal{M}(X) = \coprod_{i=0}^n \mathcal{M}(Z_i)\{a_i\}
\end{align*}
\end{thm}

A situation where the above theorem can be applied is when $X$ is a smooth projective variety with a $\mathbb{G}_m$-action. The following result is due to Iversen (\cite{iverson}), Biya{\l}nicki-Birula (\cite{bb1}, \cite{bb2}) and Hesselink (\cite{hesselink}). See Theorem 3.3 and Theorem 3.4 in \cite{bro} for more details.

\begin{thm}(\cite{bb1}, \cite{bb2}, \cite{hesselink}, \cite{iverson})
\label{thm:multaction}
Let $X$ be a smooth projective scheme over $k$ equipped with an action of  $\mathbb{G}_m$. Then,
\begin{align*}
\mathcal{M}(X) = \coprod_{i} \mathcal{M}(Z_i)\{a_i\}
\end{align*}
where $Z_i$ are connected components of $X^{\mathbb{G}_m}$ and $a_i$ are dimensions of the positive eigenspace of the action of $\mathbb{G}_m$ on the tangent space of $X$ at an arbitrary point in $Z_i$.
\end{thm}

Observe that any projective homogeneous variety over $k$ is geometrically cellular i.e., has cellular decomposition (see Definition 3.2 in \cite{bruno}) over the algebraic closure $K$ and therefore by Theorem  \ref{thm:cellular} is geometrically split i.e., its motive splits into direct sum of Tate motives over  $K$.  An important consequence of this  fact, Theorem \ref{thm:rost} and  Theorem \ref{thm:krull} is the following corollary. This is also proved by Chernousov and Merkurjev (Corollary 35 in \cite{chernousov_krull}).

\begin{cor}
\label{cor:krull_homo}
The Krull-Schmidt principle holds for any shift of any summand of the motive of  projective homogeneous varieties in $Chow(k, \Lambda)$. 
\end{cor}

The upper indecomposable motives of projective homogeneous varieties  are the basic building blocks as  proved by Karpenko in \cite{upperkarpenko}.

\begin{thm}(Theorem 3.5 in \cite{upperkarpenko})
\label{thm:karpenko}
 Let $X$ be a projective $G$-homogeneous variety. Then any indecomposable summand of $\mathcal{M}(X)$ is isomorphic to  $U_{\tau}\{i\} $ for some $i$ and some $\tau \subseteq \Delta_G$ satisfying $\tau_{k(X)} \subseteq \tau$.
\end{thm}

\section{Motivic decomposition for isotropic $G$}
\label{sec:iso}

 Recall from \cite{iverson} that for a smooth projective variety $X$ equipped with an action of $\mathbb{G}_m$, the fixed point locus $X^{\mathbb{G}_m}$ is a smooth closed subscheme of $ X$.

\begin{prop} \label{prop:surj}
Let  $X$ and $Y$ be smooth projective varieties equipped with an action of $\mathbb{G}_m$. Let $\theta: X \rightarrow Y$ be a  finite surjective $\mathbb{G}_m$-equivariant morphism.  Then the restriction morphism $\theta|_{X^{\mathbb{G}_m}} : X^{\mathbb{G}_m} \rightarrow Y^{\mathbb{G}_m}$ is surjective. 
\end{prop}
\begin{proof}
Pick a point  $y \in Y^{\mathbb{G}_m}$.  Clearly $\mathbb{G}_m$ acts on the fiber $X_y = X \times_Y Spec~k(y) $. Since  $\theta$ is finite,  $X_y$ is finite. Therefore $\mathbb{G}_m$ fixes the underlying reduced subschemes of each point in $X_y$. 
\end{proof}

\noindent A morphism $X\rightarrow Y$ of finite type is surjective if and only if the induced map $X(\Omega) \rightarrow Y(\Omega)$ is surjective for every algebraically closed field $\Omega$ (EGA IV, Chapter 1, \S6, Proposition 6.3.10).  Using this we get an easy corollary of the above proposition.

\begin{cor}
\label{cor:bijective}
With notations as in Proposition \ref{prop:surj}, let $\{X_i\}_{i=1}^n$ and   $\{Y_i\}_{i=1}^m$  denote the  connected components of $X^{\mathbb{G}_m}$ and $Y^{\mathbb{G}_m}$ respectively.  Suppose $\theta:  X(\Omega) \rightarrow Y(\Omega)$  is bijective for every algebraically closed field $\Omega$. Then $n=m$ and after permuting indices,   $\theta|_{X_i}: X_i(\Omega) \rightarrow Y_i(\Omega)$ is  also bijective.
\end{cor}

In this section we assume that $G$ is an isotropic,  semi-simple group of inner type over $k$.  We  fix an  embedding $\lambda: \mathbb{G}_m \rightarrow G$ of a $k$-split torus. Let $H$ denote the centralizer of $\lambda$ in $G$. Then by  Theorem 6.4.7  in \cite{springer},   $H$ is connected and reductive. It is defined over $k$ by Proposition 13.3.1 of \cite{springer}. Recall that  if $X_K \simeq G/P$ and  $\widetilde{X}_K \simeq G/\widetilde{P}$, we have a canonical $G$-equivariant morphism  $\theta: X \rightarrow \widetilde{X}$.

\begin{thm}
\label{thm:main}
Let $\widetilde{X}$  be a projective pseudo-homogeneous variety for $G$ and let $X$ be the corresponding projective  homogeneous variety. Then each connected component of the  fixed point locus $\widetilde{X}^\lambda$ is a projective  pseudo-homogeneous for $H$. Moreover  if $\widetilde{X}^\lambda = \coprod \widetilde{Z_i}$, then $X^{\lambda} = \coprod Z_i$ where $Z_i$ are the projective $H$-homogeneous varieties corresponding to $\widetilde{Z_i}$
\end{thm}
\begin{proof}

First note that  $H$ acts on $\widetilde{X}^\lambda$ because $\lambda(t)\cdot h \cdot  x = h \cdot \lambda(t) \cdot x = h \cdot x$ $\forall h \in H, t \in \mathbb{G}_m , x \in \widetilde{X}^{\lambda}$. Let $Y$ be a connected component of $\widetilde{X}^{\lambda}$.  It suffices to show that  the action map $H \times Y \rightarrow Y \times Y$ is surjective on $\Omega$-points for every algebraically closed field $\Omega$ over $K$. By III, \S1, 1.15 of \cite{DG}, the $G$-equivariant morphism $\theta(\Omega): X(\Omega) \xrightarrow{} \widetilde{X}(\Omega)$ is bijective. Therefore,  by Corollary \ref{cor:bijective}, ${X}^{\lambda}(\Omega) \rightarrow\widetilde{X}^{\lambda}(\Omega)$ is also  bijective.  So there exists a connected component  $Z$ of $X^{\lambda}$ such that $\theta: Z(\Omega) \rightarrow Y(\Omega)$ is a bijection. By Theorem 7.1 in  \cite{bro}, $Z$ is projective homogeneous for ${H}$. Therefore the action map $H \times Z \rightarrow Z \times Z$ is surjective on $\Omega$-points.  We have the following commutative diagram:

\begin{equation*}
\begin{tikzpicture}[thick]
  \node (A) {${H} \times Z$};
  \node (B) [right of=A] {$Z \times Z$};
  \node (C) [below of=A] {${H} \times Y$};
  \node (D) [below of=B] {$Y\times Y$};

  \draw[->] (A) to node {} (B);
  \draw[->] (A) to node [swap] {$(id, \theta)$} (C);
  \draw[->] (C) to node [swap] {} (D);
  \draw[->] (B) to node {$(\theta,\theta)$} (D);
 
\end{tikzpicture}
\end{equation*}

The morphisms given by the top  arrow and $(\theta,\theta)$ are surjective on $\Omega$-points as argued before. Hence we conclude that the bottom arrow is surjective on $\Omega$-points.  This  proves that each $Y$ is projective pseudo-homogeneous for ${H}$. \\
\indent For the second part of the claim note that if $x \in Z(K)$, then  $Stab_H(x) \subseteq Stab_H(\theta(x))$. This together with the bijectivity of $\theta: Z(K) \rightarrow Y(K)$ shows that $Z$ is the projective homogeneous variety corresponding to $Y$.
\end{proof}

We now analyze the action of $\lambda$ on the tangent space at any point in the fixed point locus $\widetilde{X}^\lambda$.    As before $X_K \simeq G/P$ and  $\widetilde{X}_K \simeq G/\widetilde{P}$. Let $b \in(G/P)^{\lambda}$.  Let $a \in G/P$ be the  unique point whose stabilizer in $G_K$ is $P$ and let $b = g\cdot a$ for some $g \in G(K)$. Then $g^{-1} \lambda g \subseteq T  \subseteq P$ for some maximal torus $T$. Let $T' = gTg^{-1}$.    Let  $\beta_1, \beta_2,\cdots, \beta_n$ be the negative roots of $G_K$ with respect to $T$ and a Borel $B$ such that $T \subseteq B \subseteq P$.  Recall from Theorem 6 in \cite{vufs} that to every parabolic  subscheme, one can associate a $W$-function defined as follows.
\begin{defn}(Definition 5 in \cite{vufs})
Write $\mathbb{N}^*$  to signify the set of non-negative natural numbers together with $\infty$. Let $\phi^+$ denote the set of positive roots of $G$. A $W$-function on $\phi^+$ is a function, $f$, on $\phi^+$  with values in $\mathbb{N}^*$ satisfying  the condition,
\begin{align*}
f(\beta) =\inf_{\alpha \in supp(\beta)} f(\alpha)
\end{align*}
 where $supp(\beta) = \{ \gamma \in \phi^+ | \beta = \gamma + \delta, \text{~for~ some~} \delta \in \phi^+\}$.
\end{defn}

\begin{remark}
\label{rmk:char}
In order to associate a $W$-function to a  parabolic subscheme as in Theorem 6 in \cite{vufs}, the authors of  the paper assume that $char~K  >3$.  This assumption  is necessary by Remark 15 in \cite{wenzel}.
\end{remark}

Let  $f$ be  the $W$-function associated to $\widetilde{P}$  and let $n_i = f(-\beta_i)$. Without loss of generality, assume that $\beta_1, \beta_2, \cdots, \beta_m$ are the negative roots such that $ f(-\beta_i) < \infty$.

\begin{lem}\label{lem:tangent}
With the notations above, there exists a $T'$-stable affine  open neighborhood  of   $\theta(b)$  in $(G/\widetilde{P})^{\lambda}$  parametrized by $T'$ - eigen functions with weights $p^{n_i} \alpha_i $ where  $\alpha_i$ are characters of $T'$. In other words, one can find an open set $V = Spec~K[X_1, X_2, \cdots, X_m]$  containing $\theta(b)$ such that
\begin{align*}
t' \cdot X_i = \alpha_i^{ p^{n_i}}(t')~ X_i~~~~~~~ \forall t' \in T'
\end{align*} 
\end{lem}

\begin{proof} Let $U_P^0$ denote the opposite of the unipotent radical of $P$. By Theorem 1 in \cite{vufs}, $U = U_P^0 \cdot \theta(a)= Spec~K[Y_1, Y_2, \cdots, Y_m]$ is an affine open neighborhood of $\theta(a)$ invariant under $T$,  where 
\begin{align*}
t \cdot Y_i = \beta_i^{ p^{n_i}}(t) ~ Y_i~~~~~~~ \forall t \in T
\end{align*}
Consider the affine open neighborhood $V = gU_P^0 \cdot \theta(a)$ of $\theta(b)$. Then 
\begin{align*}
T'\cdot V= T' gU_P^0 \cdot \theta(a) = gTU_P^0 \cdot \theta(a) =gU_P^0 \cdot \theta(a) = V
\end{align*}
So $V$ is $T'$-invariant.  Moreover $V = Spec~K[X_1, X_2, \cdots, X_m]$  where $X_i = g^{-1} \cdot Y_i$.  Let $\alpha_i$ be the character of $T'$ defined by $\alpha_i(t') = \beta_i(g^{-1} t'g) ~\forall t' \in T'$. For any point $x \in V$, write $x = gy$ where $y \in U$. Then 
\begin{align*}
t' \cdot X_i(x) &= t' \cdot (g^{-1} \cdot Y_i)(gy) =  Y_i(g^{-1} t' g y )  \\
&= \beta_i^{p^{n_i}}(g^{-1}t'g) Y_i(y)  = \alpha_i^{ p^{n_i}}(t')~ X_i(x)~~~~~~~ \forall t' \in T' 
\end{align*}

\end{proof}
 
\begin{lem}
\label{lem:twists}
For any point  $b \in X^{\lambda}$, the dimension of positive eigenspaces of the  $\lambda$-action on the tangent spaces at $b$ and  $\theta(b)$ are equal.
\end{lem}

\begin{proof}
It suffices to prove the lemma over the algebraic closure $K$ where $X_K\simeq G/P$ and $\widetilde{X}_K \simeq G/\widetilde{P}$. So assume that $k= K$.   By Lemma  \ref{lem:tangent}, there exists an affine open cover $U =Spec~K[Y_1, Y_2, \cdots, Y_m]$ of $b$ and  an affine open cover $V= Spec~K[X_1,X_2, \cdots, X_m]$  of $\theta(b)$ parametrized by $\lambda$-eigen functions with weights $\{\alpha_i\}$ and $\{p^{n_i}\alpha_i\}$ respectively.  Let $\overline{Y_i} \in \mathfrak{m}_b/ \mathfrak{m}^2_b$ and $ \overline{X_i} \in \mathfrak{m}_{\theta(b)}/ \mathfrak{m}^2_{\theta(b)}$ denote the cosets of $Y_i$ and $X_i$ respectively. Note that $\{\overline{Y_i}\}$ and $\{\overline{X_i}\}$ form a basis for  $\mathfrak{m}_b/ \mathfrak{m}^2_b$ and  $\mathfrak{m}_{\theta(b)}/ \mathfrak{m}^2_{\theta(b)}$respectively . It is now easy to see that the span of $\overline{Y_i}$ is a positive eigenspace for $\lambda$ if and only if the span of $\overline{X_i}$ is so.  By taking the dual, we are done.

\end{proof}

By  Theorem \ref{thm:multaction}, Theorem \ref{thm:main} and Lemma  \ref{lem:twists}, we get the following motivic decomposition for $\widetilde{X}$.
\begin{cor}
\label{cor:decomp1}
Let $\widetilde{X}$ and $X$ be as in  Theorem \ref{thm:main}. Then
\begin{equation*}
\mathcal{M}(\widetilde{X}) = \coprod_{i} \mathcal {M}(\widetilde{Z_i})\{a_i\}
\end{equation*}
and 
\begin{equation*}
\mathcal{M}({X}) = \coprod_{i} \mathcal {M}({Z_i})\{a_i\}
\end{equation*}
where  $\widetilde{Z_i}$ is projective pseudo-homogeneous for  $H$ and $Z_i$  is  the corresponding projective  homogeneous variety. The twists $a_i$ are dimensions of the positive eigenspace of the action of $\lambda$ on the tangent space of $X$ at an arbitrary point $z \in Z_i$.
\end{cor}

Applying  the above result  inductively, we see that each of the components in the decomposition  are  projective (pseudo-) homogeneous for the centralizer $Z(S)$ of a maximal $k$-split torus $S$. By Proposition 2.2 in \cite{gpreductifs}, we have an almost direct product decomposition $Z(S) = DZ(S) \cdot Z$ where $Z$ is the center of $Z(S)$ and $DZ(S)$ is the semi-simple anisotropic kernel. Since the center of a group is contained in every parabolic subscheme, it acts trivially on any projective pseudo-homogeneous variety. Hence, each of the $\widetilde{Z_i}$ (respectively $Z_i$) are projective pseudo-homogeneous (respectively homogeneous) for the adjoint group of the semi-simple anisotropic kernel.  Therefore we conclude:

\begin{cor}
\label{cor:decomp2}
Let $\widetilde{X}$ and $X$ be as in  Theorem \ref{thm:main}. Then
\begin{equation*}
\mathcal{M}(\widetilde{X}) = \coprod_{i} \mathcal {M}(\widetilde{Z_i})\{a_i\}
\end{equation*}
and 
\begin{equation*}
\mathcal{M}({X}) = \coprod_{i} \mathcal {M}({Z_i})\{a_i\}
\end{equation*}
where each $\widetilde{Z_i}$ (respectively $Z_i$) is either $Spec~ k$ or anisotropic projective pseudo-homogeneous (respectively homogeneous) variety for the semi-simple anisotropic kernel of $G$.  
\end{cor}

\begin{proof}
From Corollary \ref{cor:decomp1}, each $\widetilde{Z_i}$ is  projective pseudo-homogeneous variety for $H$.  Let $(\widetilde{Z_i})_K\simeq H/\widetilde{Q}$, for a parabolic subgroup scheme $\widetilde{Q}$ of $H_K$. If $\widetilde{Z_i}$ is anisotropic we are done. Suppose $\widetilde{Z_i}$ is isotropic, i.e., $\widetilde{Z_i}$ has a $k$-point. Then its stabilizer is defined over $k$ by Proposition 12.1.2 in \cite{springer}.  Without loss of generality we can assume that $\widetilde{Q}$ is defined over $k$. Since $k$ is perfect, the underlying reduced scheme $Q$ is also defined over $k$ and hence is isomorphic to $Q(\lambda)$ for some co-character $\lambda$ of $H$ defined over $k$ (Lemma 15.1.2 in \cite{springer}). So $H$ is isotropic. If $\lambda$ is a central torus, 
$Q(\lambda) = H$ and $\widetilde{Z_i} \simeq Spec ~k$. If $\lambda$ is non-central, then we can inductively use Corollary \ref{cor:decomp1} to get the result.
\end{proof}

\section{Rost Nilpotence}
\label{sec:rost}
In this section we prove that Rost nilpotence principle holds for projective pseudo-homogeneous varieties.\\

\noindent \emph{Proof of  Theorem \ref{thm:nilpotence}:}
The proof is similar to the one in \cite{bro}. For a field extension $L/k$, let $n_L$ denote the number of terms appearing in the decomposition of  Corollary \ref{cor:decomp2} for the the motive of the $G_L$-variety $\widetilde{X}_L$. Clearly, $ L \subset M  \Rightarrow  n_M \geq n_L$ and the maximal number of terms in the coproduct occurs precisely when each $\widetilde{Z_i}$ is $Spec~L$. In particular, this happens when $L =  K$.\\

\emph{Claim:} Set $N(d,n) = (d + 1)^{n_K-n}$ where $d$ is the dimension of $\widetilde{X}$. Then, for any morphism $f\in  End(M(\widetilde{X}))$ with $f \otimes K = 0$, $f^{N(d,n_k)} = 0$. \\
\indent The claim  obviously implies the theorem. Note that when $n_k = n_K$, $\mathcal{M}(\widetilde{X})$ completely splits into Tate motives and $End(\mathcal{M}(\widetilde{X})) = Ch_0(Spec~k)^{\oplus r}$ for some $r$. Therefore the claim is valid for $n_k = n_{K}$. Now we use descending induction on $n = n_k$. Let $f \in End(M(\widetilde{X}))$ be an endomorphism in the kernel of the base change map. If all components $\widetilde{Z}_i$  appearing in the motivic decomposition of  Corollary \ref{cor:decomp2} are isotropic, $n$ is maximal and the claim is already proved.  If not, pick a point $z$ in one of the anisotropic components $Z_i$ and  set $L = k(z)$. Over $L$, $ \widetilde{Z}_i$ is isotropic. Therefore, the number $n_i=n_L$ of terms appearing in the motivic decomposition of $\widetilde{X}_L$ is strictly greater than $n$. Thus the claim holds for $\mathcal{M}(\widetilde{X}_L)$ and $f_L^{N(d,n_i)} =0$. Since $N(d,n_i) \leq N(d, n+1)$, it follows that $f_L^{ N(d, n+1)} = 0$. Now we use Theorem 3.1  in  \cite{brosnan_nilpotence} to conclude that  the composition 
\begin{equation*}
\mathcal{M}(\widetilde{Z}_i)\{a_i\} \xrightarrow{j_1} \mathcal{M}(\widetilde{X}) \xrightarrow{f^{(d+1)N(d, n+1)}} \mathcal{M}(\widetilde{X})
\end{equation*}
vanishes where the first arrow comes from the coproduct decomposition. Since for each  summand the composition is zero, we are done.

\section{Split Case}
\label{sec:split}
In this section we assume that $G$ is split, so that $\widetilde{X} \simeq G/\widetilde{P}$ and $X \simeq G/P$. The goal of this section is to understand the cellular structure of $G/\widetilde{P}$ and compute its motive.

\begin{lem}
\label{lem:cellular}
 $\widetilde{X}$ is a cellular variety i.e., it has decomposition into affine cells. Moreover, the affine cells can be obtained by the image of the Schubert cells in $ G/P$ under $\theta: G/P \rightarrow G/\widetilde{P}$  .
\end{lem}
\begin{proof} We follow the proof of  \S 2.2 in \cite{lauritzen}.  We know that $X=G/P$ is cellular because $G/P$ is a disjoint union of Schubert cells $C(w) = UwP/P$ where $U$ is the unipotent radical of $B$.  Let $X(w) = \overline{C(w)}$  be the corresponding  Schubert variety. Let $\widetilde{X}(w)$ be the scheme theoretic image of $X(w)$ in $\widetilde{X} = G/\widetilde{P}$ under the canonical map $\theta: G/P \rightarrow G/\widetilde{P}$.  Call it a Schubert variety in  $\widetilde{X}$ . We get a filtration $ \widetilde{X} = \widetilde{X}_0 \supseteq \widetilde{X}_1 \supseteq  \widetilde{X}_2 \supseteq \ldots $   where $\widetilde{X}_i$ is the union of codimension $i$ Schubert varieties in $\widetilde{X}$ and  $\widetilde{X}_i - \widetilde{X}_{i+1} =  \coprod \theta(C(w))$. Here $\theta(C(w))$ are disjoint because $\theta$ is bijective.  Moreover $\theta$ is $U$-equivariant and $U$ acts transitively on $\theta(C(w))$. Therefore  by  IV.3.16  in \cite{DG}, $\theta(C(w))$ is affine. So $\widetilde{X}$ is a disjoint union of affine cells $\theta(C(w))$.
\end{proof}

\begin{lem} 
\label{lem:chow}
Wtih the notations in the proof of  Lemma \ref{lem:cellular}, the  classes  of Schubert varieties $[\widetilde{X}(w)]$ form a basis for the Chow group of $G/\widetilde{P}$. As a consequence  $Ch_i(G/\widetilde{P}) \simeq Ch_i(G/P)$.
\end{lem}
\begin{proof}
By  Example 1.9.1 in \cite{fulton}, it is clear that  the classes of Schubert varieties  $[\widetilde{X}(w)]$ form a basis for $Ch_*(G/\widetilde{P})$ and we get an isomorphism 
\begin{align*}
Ch_*(G/P) &\rightarrow Ch_*(G/\widetilde{P}) \\
[X(w)] &\rightarrow [\widetilde{X}(w)]
\end{align*}

\end{proof}

\begin{thm}
\label{thm:rank} The motive $\mathcal{M}(\widetilde{X})$ is  split i.e., it decomposes into direct sum of Tate motives.
Moreover, $\mathcal{M}(X) \simeq \mathcal{M}(\widetilde{X})$. 
\end{thm}
\begin{proof} 
This follows directly from Corollary \ref{cor:decomp2}. Alternatively, one can also argue as follows. The fact that $\mathcal{M}(\widetilde{X})$ splits into Tate motives follows by Lemma \ref{lem:cellular}, and Theorem \ref{thm:cellular}.  Now observe that for any variety whose motive splits into Tate motives,  the rank of  the $i^{th}$ Chow group  is equal to the number of summands  isomorphic to $\Lambda\{i\}$. Therefore  by Lemma \ref{lem:chow}, $\mathcal{M}(X) \simeq \mathcal{M}(\widetilde{X})$.
\end{proof}

\section{Non-Split Case}
\label{sec:final}
Now we remove the assumption that $G$ is split but keep the assumption that it is inner over $k$. In this section we show that, the motive of any projective pseudo-homogeneous variety for $G$ is same as the corresponding projective homogeneous variety. Recall the following well know fact about parabolic subgroups (\cite{boulder}). 

\begin{fact}
\label{fact:tits}
  Let $G$ be a semi-simple algebraic group over a field $k$. Let $P$ be a parabolic subgroup corresponding to subset $\tau$ of nodes of the Dynkin diagram (See \S \ref{sec:homog}). Let $\mathcal{P}$ denote the conjugacy class of $P$.  Then $\mathcal{P}$ contains a parabolic subgroup defined over $k$  if and only if the nodes in $\tau$ are circled in the Tits index of $G$ over $k$ and $\tau$ is invariant under the $*$-action of $Gal(K/k)$. 
\end{fact}

\noindent In our case, since $G$ is assumed to be  inner over $k$, the $*$-action is trivial. Let $X$ and $\widetilde{X}$ be as before.

\begin{lem}
\label{lem:point}
Let $F$ be any field extension of $k$. Then $X$ has an $F$-point iff $\widetilde{X}$ has an $F$-point.
\end{lem}
\begin{proof}
Clearly if $X$ has an $F$-point, its image via the canonical map $X \rightarrow \widetilde{X}$ gives an $F$-point on  $\widetilde{X}$. Now assume that $\widetilde{X}$ has an $F$-point.  Let $F'$ be the perfect closure of $F$. Then by Proposition 12.1.2 of \cite{springer} the stabilizer in $G$ of this $F$-point is defined over $F'$. Without loss of generality we can assume that $\widetilde{P}$ is defined over $F'$. Since $F'$ is perfect the underlying reduced subscheme $P$ is also defined over $F'$. Let $\tau$ be the subset of nodes of Dynkin diagram corresponding to $P$. Since $G$ is inner over $k$, it is inner over $F$. Therefore the  $*$-action is trivial over $F$. Moreover, by Exercise 13.2.5 (4) in \cite{springer}, the Tits index of $F'$ and $F$ are the same. Therefore by Fact \ref{fact:tits}, the conjugacy class  $\mathcal{P}$ of $P$ contains an $F$-defined parabolic  and therefore  $X$ has an $F$-point.
\end{proof}

Note that by Theorem \ref{thm:krull_pseudo}, the motive $\mathcal{M}(\widetilde{X})$ satisfies the  Krull-Schmidt principle. Therefore we can talk about \emph{the unique} upper summand $U_{\widetilde{X}}$ of   $\mathcal{M}(\widetilde{X})$.

\begin{cor}
\label{cor:upper}
Let $X$ and $\widetilde{X}$ be as above. Then  in $Chow(k, \Lambda)$, $U_X \simeq U_{\widetilde{X}}$.
\end{cor}
\begin{proof}
By Corollary 2.15 of \cite{upperkarpenko},  it suffices to show multiplicity one correspondences $\alpha: \mathcal{M}(X) \rightarrow \mathcal{M}(\widetilde{X})$ and $\beta: \mathcal{M}(\widetilde{X}) \rightarrow \mathcal{M}(X)$. Take $\alpha$ to be the correspondence induced from the canonical map $X \rightarrow \widetilde{X}$.  For $\beta$, first observe that $\widetilde{X}$ has an $k(\widetilde{X})$-point. Then by Lemma \ref{lem:point}, so does $X$. Now take $\beta$ to be the correspondence induced from the rational map $\widetilde{X} \dashedrightarrow X$.

\end{proof}

\noindent \emph{Notation:} For a variety $X$, $A^i(X, \Lambda)$  denotes the $i^{th}$ Chow group of $X$ with coefficients in $\Lambda$ graded   by codimension.  We simply write $A^i$  if $X$ and $\Lambda$ are clear from the context. $A^{\geq i}$ denotes $\bigoplus_{j \geq i } A^j$. Similarly define $A^{> i}$,  $A^{\leq i}$  and  $A^{< i}$.  \\
\indent  $A_i(X, \Lambda)$  denotes the $i^{th}$ Chow group of $X$ with coefficients in $\Lambda$ graded   by dimension. We make similar definitions  for $A_{\geq i}$, $A_{ >i}$,  $A_{\leq i}$  and  $A_{< i}$. \\

Recall  that for a motive $M$, $Ch^i(M)$ is defined as  $Hom(M, \Lambda\{i\})$  in the category $Chow(k, \Lambda)$.\\
\noindent\emph{Definition:}  Let $\epsilon$ be the function  on the objects of $Chow(k, \Lambda)$ defined as follows:
\begin{align*}
\epsilon: Ob(Chow(k, \Lambda)) &\longrightarrow ~~~~~~\mathbb{N} ~\bigcup ~ \{-\infty\}\\
M ~~~~~~~&\longmapsto min\{i ~| Ch^i(M_{K}) \neq 0\}
\end{align*}

\subsection{Proof of Theorem \ref{thm:cool}}
\label{sec:final1}
\indent Since $X$ is projective homogeneous variety for $G$, by Theorem 1.1 of \cite{outerkarpenko}, every indecomposable summand $M$ of $\mathcal{M}(X)$ is isomorphic to  $U_{Y}\{i\}$ for some projective homogeneous variety $Y$ corresponding to  $\tau$ such that $\tau \supseteq \tau_L$. By condition (2),  $U_{Y_L}\{i\}$ comes from  an indecomposable summand $\widetilde{M}$ of  $\mathcal{M}(Z)$ (Here $U_{Y_L}$ denotes the upper motive of $Y_L$. It is not the same as $(U_Y)_L$. But is the upper motive of  $(U_Y)_L$). We claim that $M \simeq \widetilde{M}$.  It is clear that if $M$ and $N$ are distinct (may or may not be isomorphic) indecomposable summands of $\mathcal{M}(X)$, $\widetilde{M}$ and $\widetilde{N}$ are distinct  indecomposable  summands of $\mathcal{M}(\widetilde{X})$.  This together with condition (2)   implies that  it suffices to prove the claim to complete the proof. \\
 \indent The proof of the claim is by induction on $\epsilon (M)$. For the base case $\epsilon (M) = 0$, the claim clearly holds by condition (1).  Now let $M \simeq U_Y\{i\}$ be a summand of $\mathcal{M}(X)$ as above. Then $\epsilon(M) = i$ and assume that for all indecomposable summands $N$ with $\epsilon(N) < i$, $N \simeq \widetilde{N}$. Write $\mathcal{M}(X) = P \bigoplus Q$ where $\epsilon(P') < i$ for every indecomposable summand $P'$ of $P$ and $\epsilon(Q) \geq i$ . Then by induction hypothesis, $\mathcal{M}(Z) \simeq P \bigoplus R$. By Theorem \ref{thm:krull},  $Q_L \simeq R_L$. By assumption $M$ is a summand of $Q$ and so $\widetilde{M}$ is a summand of $R$. Observe that $\epsilon(\widetilde{M}_L) = i$ as $\epsilon(Q_L) \geq i$. Therefore if $\pi \in End(\mathcal{M}(Z))$ is the projector giving rise to the summand $\widetilde{M}$, then $\pi_{\overline{L}} = \sum b_k \times a_k \in \sum_{I} A^r \times A_r$ for a multiset $I$ such that  $r \geq i$ for every $r \in I$ and $a_k \cdot b_j = \delta_{kj}$ (Here $\delta_{kj}$ is the Kronecker delta function). \\
\indent To complete the proof, it suffices to find  $\alpha: \mathcal{M}(Y)\{i\} \longrightarrow \widetilde{M}$ and $\beta: \widetilde{M} \longrightarrow \mathcal{M}(Y)\{i\}$ such that $mult(\beta \circ \alpha) = 1$ (See Lemma 2.14 of \cite{upperkarpenko}). \\
\indent For a motive $N$ over $k$, let $\overline{N}$ denote the motive base changed to $\overline{L}$ and for a variety $V$ over $k$, $\overline{V}$ denotes $V \times_{Spec~k} \overline{L}$. \\
\indent First note that we have $a \in Hom(\Lambda\{i\}, \mathcal{M}(\overline{Z})) = A_i(\overline{Z})$ given by $\Lambda\{i\} \hookrightarrow \overline{U}_{Y_L}\{i\} \hookrightarrow \mathcal{M}(\overline{Z})$  and   $b \in Hom(\mathcal{M}(\overline{Z}), \Lambda\{i\}) = A^i(\overline{Z})$ given by $\mathcal{M}(\overline{Z}) \rightarrow \overline{U}_{Y_L}\{i\} \rightarrow \Lambda\{i\}$ such that  $mult(b \circ a) =1$  i.e., $a \cdot b = 1$. Observe that with this notation, $\overline{\pi} = b \times a + \sum_{k} b_k \times a_k$ where $b_k \times a_k \in A^{\geq i} \times A_{\geq i}$, $a \cdot b_k = 0 ~\forall b_k$ and $a_k\cdot b = 0~ \forall a_k$. \\
\noindent \underline{\emph{Construction of $\alpha$}}: \\
\indent Let $\alpha _1 \in Hom(\mathcal{M}(Y_L)\{i\}, \mathcal{M}(Z_L)) = A^{dim~Z-i}(Y_L \times Z_L)$  be given by $\mathcal{M}(Y_L)\{i\} \rightarrow U_{Y_L}\{i\} \hookrightarrow \mathcal{M}(X_L) \xrightarrow{\simeq} \mathcal{M}(Z_L)$. Then 
\begin{align*}
\overline{\alpha}_1 \in 1 \times a + A^{>0} \times A_{>i}
\end{align*}

Let $\alpha_2$ be the image of $\alpha_1$ under the pull back of Chow groups
\begin{align*}
A^{dim~Z -i}(Y_L \times_L Z_L) \longrightarrow A^{dim~Z -i} (Spec~L(Y) \times_L Z_L)
\end{align*}
induced by $Spec~L(Y) \times_L \times Z_L \rightarrow Y_L \times_L Z_L \simeq (Y \times Z)_L$. Then 
\begin{align*}
\overline{\alpha}_2  = Spec~\overline{L}(Y) \times a.
\end{align*}
 
Since $\tau_L \subseteq \tau$, $X$ has an $k(Y)$-point. So $k(Y)(X)/k(Y) = L(Y)/k(Y)$ is purely transcendental. Therefore $\alpha_2$ is $k(Y)$ rational. So $\alpha_2 \in A^{dim~Z -i}(Spec~ k(Y) \times Z)$. Let $\alpha'$ be any preimage of $\alpha_2$ under the surjective map  of Chow groups
\begin{align*}
A^{dim~Z -i }(Y \times Z) 	\twoheadrightarrow A^{dim~Z -i}(Spec~k(Y) \times Z)
\end{align*}
induced by $Spec~k(Y) \times Z \rightarrow Y \times Z$. Then 
\begin{align*}
\overline{\alpha'} \in 1 \times a + A^{>0} \times A_{>i}
\end{align*}

Let $p:\mathcal{M}(Z) \rightarrow \widetilde{M}$ be the projection from our decomposition. Define 
\begin{align*}
\alpha = p \circ \alpha'
\end{align*}

\noindent \underline{\emph{Construction of $\beta$}}:\\
\indent Let $\beta_1 \in Hom( \mathcal{M}(Z_L), \mathcal{M}(Y_L)\{i\})$ be given by $\mathcal{M}(Z_L) \xrightarrow{\simeq} \mathcal{M}(X_L) \rightarrow  U_{Y_L}\{i\} \rightarrow \mathcal{M}(Y_L)\{i\}$. Then,
\begin{align*}
\overline{\beta}_1 \in b \times y+ A^{>i} \times A_{>0}
\end{align*}
where $y$ is the class of a point in $\overline{Y}$.
Let $\beta_2$ be an element in the inverse image of $\beta_1$ under the surjective map of Chow groups
\begin{align*}
A^{dim~Y + i}(Z \times X \times Y) \twoheadrightarrow A^{dim~Y +i}(Z_L \times Y_L)
\end{align*}
induced by $Z_L \times_{L} Y_L \simeq (Z \times_k Y) \times Spec~k(X)  \rightarrow Z \times Y \times X \rightarrow Z \times X \times Y$ where the last  map is  obtained by switching second and third factors.  Then

\begin{align*}
\overline{\beta}_2  \in b \times 1 \times y + A^{>i} \times 1 \times A_{>0} +  A^* \times  A^{>0} \times  A^*
\end{align*}

Recall that $\pi\in End(\mathcal{M}(Z))$ is the projector giving  the summand $\widetilde{M}$.  Let $\beta_3 = \beta_2 \circ \pi$ where $\beta_2$ is thought of as an element in $Hom(\mathcal{M}(Z), \mathcal{M}(X \times Y)\{i - dim~X\})$.  Then

\begin{align*}
\overline{\beta}_3 \in p_{134*}[(b \times a \times 1 \times 1 +  \sum_{k} b_k &\times a_k \times 1 \times 1) \cdot (1 \times b \times 1 \times y + 1 \times A^{>i} \times 1 \times A_{>0} + 1 \times A^* \times A^{>0} \times A^*)]\\
  &\overline{\beta}_3  \in  b \times 1 \times y +  A^{i} \times A^{>0} \times A^* +  A^{>i} \times 1 \times A_{>0}  + A^{>i} \times A^{>0} \times A^*
\end{align*}

By condition (1) in the hypothesis of the theorem $U_X \simeq U_Z$. This implies by Corollary 2.15 of \cite{upperkarpenko} that we have a multiplicity 1 correspondence $\Gamma \in A_{dim~Z}(Z \times X)$ . Then $\overline{\Gamma} = 1 \times x + A^{>0} \times A_{>0}$ where $x$ refers to the class of a point in $\overline{X}$. \\
\indent Now $ \Gamma \times 1 \in A_{dim~Z + dim~ Y}(Z \times X \times Y)$. Define  $\beta' = p_{13*}[(\Gamma \times 1 ) \cdot \beta_3] \in A^{dim~Y +i}(Z \times Y) = Hom(\mathcal{M}(Z), \mathcal{M}(Y)\{i\})$. Then

\begin{align*}
\overline{\beta'} \in p_{13*}[(1 \times x \times 1 + A^{>0} \times A_{>0} \times 1) &\cdot (b \times 1 \times y +  A^{i} \times A^{>0} \times A^* +  A^{>i} \times 1 \times A_{>0}+  A^{>i} \times A^{>0} \times A^*)]\\
\overline{\beta'}  &\in b \times y + A^{>i} \times A_{>0}
\end{align*}

Now define $\beta = \beta' \circ q$ where $q: \widetilde{M} \hookrightarrow \mathcal{M}(Z)$ is inclusion map from our decomposition.  \\
\indent We now see that $\beta \circ \alpha = \beta' \circ q \circ p \circ \alpha' = \beta' \circ \pi \circ \alpha'$. Note that 

\begin{align*}
\overline{\pi} \circ \overline{\alpha'}  \in p_{13*}[(1 \times a &\times 1 + A^{>0} \times A_{>i} \times 1) \cdot (1 \times b \times a + \sum_{k}1 \times b_k \times a_k)]\\
& \overline{\pi} \circ \overline{\alpha'}   \in 1 \times a + A^{>0} \times A_{>i}
\end{align*}

Finally we see that 
\begin{align*}
\overline{\beta} \circ \overline{\alpha} \in p_{13*}[( 1 \times a &\times 1  + A^{>0} \times A_{>i} \times 1) \cdot (1 \times b \times y +1 \times  A^{>i} \times A_{>0})]\\
&\overline{\beta} \circ \overline{\alpha} \in 1  \times y + A^{>0} \times A_{>0}
\end{align*}

Therefore,  $mult(\beta \circ \alpha) =1$.
\\

We are now ready to prove Theorem \ref{thm_motives}. \\

\subsection{Proof of Theorem \ref{thm_motives}}
\label{sec:final2}
\indent We will prove by induction on $n= \mathrm{rank}(G)$. The claim is trivially true for $n=0$. Assume that the claim is true for all groups with rank less than  $n$. Let $\mathrm{rank}(G) = n$. We can assume that $X \neq Spec(k)$ (otherwise there is nothing to prove).  Let $L=k(X)$ and $G'$ the anisotropic kernel of $G_L$. Then $\mathrm{rank}(G') < \mathrm{rank}(G)$.  Now by Corollary \ref{cor:decomp2}, $\mathcal{M}(\widetilde{X}_L) = \coprod_{i} \mathcal {M}(\widetilde{Z_i})\{a_i\}$  and $\mathcal{M}(X_L) = \coprod_{i} \mathcal {M}(Z_i)\{a_i\}$ where $\widetilde{Z_i}$ is projective pseudo-homogeneous for $G'$ and $Z_i$ the corresponding projective homogeneous variety.  By induction hypothesis, we have  $\mathcal {M}(\widetilde{Z_i}) \simeq  \mathcal {M}({Z_i})$ and thus  $\mathcal{M}(\widetilde{X}_L) \simeq \mathcal{M}(X_L)$. Moreover by Corollary \ref{cor:upper} $U_X \simeq U_{\widetilde{X}}$. Therefore, by Theorem \ref{thm:cool}, we are done.

\section*{Acknowledgments}

I would like to thank my advisor, Patrick Brosnan, for introducing  this topic to me and suggesting this problem for my thesis. I am very grateful to him for all the  useful discussions and for his feedback as it greatly improved the quality of this manuscript.

\nocite*{}
\bibliographystyle{plain}
\bibliography{reference}

\begin{thebibliography}{10}

\bibitem{artin}
M.~Artin.
\newblock Brauer-{S}everi varieties.
\newblock In {\em Brauer groups in ring theory and algebraic geometry
  ({W}ilrijk, 1981)}, volume 917 of {\em Lecture Notes in Math.}, pages
  194--210. Springer, Berlin-New York, 1982.

\bibitem{bb1}
A.~Bia{\l}ynicki-Birula.
\newblock Some theorems on actions of algebraic groups.
\newblock {\em Ann. of Math. (2)}, 98:480--497, 1973.

\bibitem{bb2}
A.~Bia{\l}ynicki-Birula.
\newblock Some properties of the decompositions of algebraic varieties
  determined by actions of a torus.
\newblock {\em Bull. Acad. Polon. Sci. S\'er. Sci. Math. Astronom. Phys.},
  24(9):667--674, 1976.

\bibitem{gpreductifs}
Armand Borel and Jacques Tits.
\newblock Compl\'ements \`a l'article: ``{G}roupes r\'eductifs''.
\newblock {\em Inst. Hautes \'Etudes Sci. Publ. Math.}, (41):253--276, 1972.

\bibitem{brosnan_nilpotence}
Patrick Brosnan.
\newblock A short proof of {R}ost nilpotence via refined correspondences.
\newblock {\em Doc. Math.}, 8:69--78, 2003.

\bibitem{bro}
Patrick Brosnan.
\newblock On motivic decompositions arising from the method of {B}ia\l
  ynicki-{B}irula.
\newblock {\em Invent. Math.}, 161(1):91--111, 2005.

\bibitem{calmes}
B.~Calm{\`e}s, V.~Petrov, N.~Semenov, and K.~Zainoulline.
\newblock Chow motives of twisted flag varieties.
\newblock {\em Compos. Math.}, 142(4):1063--1080, 2006.

\bibitem{chernousov_krull}
V.~Chernousov and A.~Merkurjev.
\newblock Motivic decomposition of projective homogeneous varieties and the
  {K}rull-{S}chmidt theorem.
\newblock {\em Transform. Groups}, 11(3):371--386, 2006.

\bibitem{chernosov_iso}
Vladimir Chernousov, Stefan Gille, and Alexander Merkurjev.
\newblock Motivic decomposition of isotropic projective homogeneous varieties.
\newblock {\em Duke Math. J.}, 126(1):137--159, 2005.

\bibitem{delBano}
Sebastian del Ba{\~n}o.
\newblock On the {C}how motive of some moduli spaces.
\newblock {\em J. Reine Angew. Math.}, 532:105--132, 2001.

\bibitem{De77}
M.~Demazure.
\newblock Automorphismes et d\'eformations des vari\'et\'es de {B}orel.
\newblock {\em Invent. Math.}, 39(2):179--186, 1977.

\bibitem{DG}
Michel Demazure and Pierre Gabriel.
\newblock {\em Groupes alg\'ebriques. {T}ome {I}: {G}\'eom\'etrie alg\'ebrique,
  g\'en\'eralit\'es, groupes commutatifs}.
\newblock Masson \& Cie, \'Editeur, Paris; North-Holland Publishing Co.,
  Amsterdam, 1970.
\newblock Avec un appendice {{\i}t Corps de classes local} par Michiel
  Hazewinkel.

\bibitem{elman_quadratic}
Richard Elman, Nikita Karpenko, and Alexander Merkurjev.
\newblock {\em The algebraic and geometric theory of quadratic forms},
  volume~56 of {\em American Mathematical Society Colloquium Publications}.
\newblock American Mathematical Society, Providence, RI, 2008.

\bibitem{florence}
Mathieu Florence.
\newblock On the symbol length of {$p$}-algebras.
\newblock {\em Compos. Math.}, 149(8):1353--1363, 2013.

\bibitem{fulton}
William Fulton.
\newblock {\em Intersection theory}, volume~2 of {\em Ergebnisse der Mathematik
  und ihrer Grenzgebiete. 3. Folge. A Series of Modern Surveys in Mathematics
  [Results in Mathematics and Related Areas. 3rd Series. A Series of Modern
  Surveys in Mathematics]}.
\newblock Springer-Verlag, Berlin, second edition, 1998.

\bibitem{gille_book}
Petrov V. Semenov~N. Gille, S. and Zainoulline K.
\newblock Introduction to motives and algebraic cycles on projective
  homogeneous varieties.

\bibitem{vufs}
William Haboush and Niels Lauritzen.
\newblock Varieties of unseparated flags.
\newblock In {\em Linear algebraic groups and their representations ({L}os
  {A}ngeles, {CA}, 1992)}, volume 153 of {\em Contemp. Math.}, pages 35--57.
  Amer. Math. Soc., Providence, RI, 1993.

\bibitem{hesselink}
Wim~H. Hesselink.
\newblock Concentration under actions of algebraic groups.
\newblock In {\em Paul {D}ubreil and {M}arie-{P}aule {M}alliavin {A}lgebra
  {S}eminar, 33rd {Y}ear ({P}aris, 1980)}, volume 867 of {\em Lecture Notes in
  Math.}, pages 55--89. Springer, Berlin, 1981.

\bibitem{iverson}
Birger Iversen.
\newblock A fixed point formula for action of tori on algebraic varieties.
\newblock {\em Invent. Math.}, 16:229--236, 1972.

\bibitem{bruno}
Bruno Kahn.
\newblock Motivic cohomology of smooth geometrically cellular varieties.
\newblock In {\em Algebraic {$K$}-theory ({S}eattle, {WA}, 1997)}, volume~67 of
  {\em Proc. Sympos. Pure Math.}, pages 149--174. Amer. Math. Soc., Providence,
  RI, 1999.

\bibitem{cellular_karpenko}
N.~A. Karpenko.
\newblock Cohomology of relative cellular spaces and of isotropic flag
  varieties.
\newblock {\em Algebra i Analiz}, 12(1):3--69, 2000.

\bibitem{karpenko_mequiv}
Nikita~A. Karpenko.
\newblock Criteria of motivic equivalence for quadratic forms and central
  simple algebras.
\newblock {\em Math. Ann.}, 317(3):585--611, 2000.

\bibitem{outerkarpenko}
Nikita~A. Karpenko.
\newblock Upper motives of outer algebraic groups.
\newblock In {\em Quadratic forms, linear algebraic groups, and cohomology},
  volume~18 of {\em Dev. Math.}, pages 249--257. Springer, New York, 2010.

\bibitem{upperkarpenko}
Nikita~A. Karpenko.
\newblock Upper motives of algebraic groups and incompressibility of
  {S}everi-{B}rauer varieties.
\newblock {\em J. Reine Angew. Math.}, 677:179--198, 2013.

\bibitem{karpenko_pseudo}
Nikita~A. Karpenko.
\newblock Incompressibility of products of pseudo-homogeneous varieties.
\newblock {\em Canad. Math. Bull.,}, 2016.
\newblock \url{ http://dx.doi.org/10.4153/CMB-2016-024-4}.

\bibitem{brauer_p}
M.~A. Knus, M.~Ojanguren, and D.~J. Saltman.
\newblock On {B}rauer groups in characteristic {$p$}.
\newblock In {\em Brauer groups ({P}roc. {C}onf., {N}orthwestern {U}niv.,
  {E}vanston, {I}ll., 1975)}, pages 25--49. Lecture Notes in Math., Vol. 549.
  Springer, Berlin, 1976.

\bibitem{boi}
Max-Albert Knus, Alexander Merkurjev, Markus Rost, and Jean-Pierre Tignol.
\newblock {\em The book of involutions}, volume~44 of {\em American
  Mathematical Society Colloquium Publications}.
\newblock American Mathematical Society, Providence, RI, 1998.
\newblock With a preface in French by J. Tits.

\bibitem{splitting_lauritzen}
Niels Lauritzen.
\newblock Splitting properties of complete homogeneous spaces.
\newblock {\em J. Algebra}, 162(1):178--193, 1993.

\bibitem{lauritzen_embedding}
Niels Lauritzen.
\newblock Embeddings of homogeneous spaces in prime characteristics.
\newblock {\em Amer. J. Math.}, 118(2):377--387, 1996.

\bibitem{lauritzen}
Niels Lauritzen.
\newblock Schubert cycles, differential forms and {$\mathcal{D}$}-modules on
  varieties of unseparated flags.
\newblock {\em Compositio Math.}, 109(1):1--12, 1997.

\bibitem{index_red}
A.~S. Merkurjev, I.~A. Panin, and A.~R. Wadsworth.
\newblock Index reduction formulas for twisted flag varieties. {I}.
\newblock {\em $K$-Theory}, 10(6):517--596, 1996.

\bibitem{AGSMilne}
J.S. Milne.
\newblock Basic theory of affine group schemes.
\newblock \url{http://www.jmilne.org/math/CourseNotes/AGS.pdf}.

\bibitem{rost_pfister}
Markus Rost.
\newblock The motive of a {P}fister form.
\newblock \url{https://www.math.uni-bielefeld.de/~rost/motive.html}.

\bibitem{desalas}
Carlos Sancho~de Salas.
\newblock Complete homogeneous varieties: structure and classification.
\newblock {\em Trans. Amer. Math. Soc.}, 355(9):3651--3667 (electronic), 2003.

\bibitem{springer}
T.~A. Springer.
\newblock {\em Linear algebraic groups}.
\newblock Modern Birkh\"auser Classics. Birkh\"auser Boston, Inc., Boston, MA,
  second edition, 2009.

\bibitem{boulder}
J.~Tits.
\newblock Classification of algebraic semisimple groups.
\newblock In {\em Algebraic {G}roups and {D}iscontinuous {S}ubgroups ({P}roc.
  {S}ympos. {P}ure {M}ath., {B}oulder, {C}olo., 1965)}, pages 33--62. Amer.
  Math. Soc., Providence, R.I., 1966, 1966.

\bibitem{wenzel}
Christian Wenzel.
\newblock Classification of all parabolic subgroup-schemes of a reductive
  linear algebraic group over an algebraically closed field.
\newblock {\em Trans. Amer. Math. Soc.}, 337(1):211--218, 1993.

\bibitem{rationality}
Christian Wenzel.
\newblock Rationality of {$G/P$} for a nonreduced parabolic subgroup-scheme
  {$P$}.
\newblock {\em Proc. Amer. Math. Soc.}, 117(4):899--904, 1993.

\end{thebibliography}

\end{document}